\documentclass[11pt]{amsart}

\usepackage{amsthm,amscd,amssymb,amsmath,tikz-cd,mathtools,enumerate}
\usepackage[mathscr]{eucal}

\usepackage{lmodern}
\usepackage[T1]{fontenc}

\theoremstyle{defn}
\newtheorem{defn}{Definition}[section]

\theoremstyle{plain}
\newtheorem{lemma}[defn]{Lemma}
\newtheorem{theorem}[defn]{Theorem}
\newtheorem{proposition}[defn]{Proposition}
\newtheorem{corollary}[defn]{Corollary}
\newtheorem{conjecture}[defn]{Conjecture}

\theoremstyle{remark}
\newtheorem{remark}[defn]{Remark}
\newtheorem{example}[defn]{Example}

\newcommand{\HH}{{H}}
\newcommand{\T}{\mathbb{T}}
\newcommand{\rank}{{\mathrm {rank} \,}}

\newcommand{\Gr}{\mathrm{Gr}}
\newcommand{\B}{\mathrm{B}}

\usepackage[margin=1in]{geometry}

\begin{document}

\title{{Torus fibers and the weight filtration}}

\author{Andrew Harder}
	
	\address{Andrew Harder,
	\textnormal{Department of Mathematics, Lehigh University, Christmas-Saucon Hall, 14 E. Packer Ave., Bethlehem, PA, USA, 18015.}
	\textnormal{\texttt{anh318@lehigh.edu}}}

\begin{abstract}
We show that if $(X,Y)$ is a simple normal crossings log Calabi--Yau pair, then there is a real torus of dimension equal to the codimension of the smallest stratum of $Y$ which can be used to construct $W_{2k-1}H^k(X \setminus Y;\mathbb{Q})$ for all $k$. We show that an analogous result holds for degenerations of Calabi--Yau varieties. We use this to show that P=W type results hold for pairs $(X,Y)$ consisting of a rational surface $X$ and a nodal anticanonical divisor $Y$, and for K3 surfaces.
\end{abstract}

\maketitle

\tableofcontents
\section{Introduction}

In this paper, we will deal with mixed Hodge structures associated to geometric data of the following types.
\begin{enumerate}
\item $(X,Y)$ where $X$ is smooth and projective, and where $Y$ is simple normal crossings  in $X$.
\item $(\mathscr{X},\pi)$ where $\mathscr{X}$ is K\"ahler, $\pi$ is a projective map to the unit disc, $X_0 = \pi^{-1}(0)$ is simple normal crossings, and all other fibers are smooth.
\end{enumerate}
In the first case, we associate the mixed Hodge structure on $H^*(X\setminus Y;\mathbb{Q})$ (following Deligne \cite{deligne2, deligne}), and in the second case we associate the limit mixed Hodge structure (following Steenbrink \cite{steen} or Schmid \cite{schmid}) whose underlying vector space we take to be $H^*(X_1;\mathbb{Q})$ where $X_1 = \pi^{-1}(1).$ 

The goal of this note is to show there are naturally defined real tori in $X \setminus Y$ (resp.  $X_1$) which can be used to compute the highest part of the weight filtration in the relevant mixed Hodge structure, and to explore the consequences of this fact. We are particularly interested in the following specializations of the situations listed above.
\begin{enumerate}
\item $Y$ anticanonical in $X$ (in which case, we say that $(X,Y)$ is {\em log Calabi--Yau}).
\item $K_{\mathscr{X}}$ is trivial (in which case we say that $(\mathscr{X},\pi)$  is {\em Calabi--Yau}). 
\end{enumerate}
Our main technical result is the following theorem.
\begin{theorem}[Theorem \ref{thm:main}, Theorem \ref{thm:maincy}]\label{thm:mainintro}
Let $\delta$ be the maximal number of components of $Y$ (resp $X_0$) which intersect nontrivially.  
\begin{enumerate}
\item
If $(X,Y)$ is log Calabi--Yau, then there is a torus $\mathbb{T}$ of real dimension $\delta$ in $X \setminus Y$ so that for all $k$,
\[
W_{2k-1} H^k(X\setminus Y;\mathbb{Q}) = \mathrm{ker}\left(H^k(X \setminus Y;\mathbb{Q}) \longrightarrow H^k(\T;\mathbb{Q})\right).
\]
\item
If $(\mathscr{X},\pi)$ is Calabi--Yau, then there is a torus $\mathbb{T}$ of real dimension $\delta-1$ in $X_1$ so that for all $k$,
\[
W_{2k-1} H^k(X_1;\mathbb{Q}) = \mathrm{ker}\left( H^k(X _1;\mathbb{Q}) \longrightarrow H^k(\T;\mathbb{Q})\right).
\]
\end{enumerate}
\end{theorem}
Let us now describe the torus $\T$ referred to in the statement of Theorem \ref{thm:mainintro}(1). Assume that $Y$ contains a point $p$ which is the intersection of $\delta$ divisors. We can choose local coordinates $z_1,\dots, z_d$ centered at $p$ so that $Y = Z(z_1\dots z_\delta)$. Then for some constant $0  < \varepsilon \ll 1$,
\[
\T = \{ (\varepsilon\exp(\mathtt{i}\theta_1), \dots, \varepsilon \exp(\mathtt{i}\theta_\delta), 0,\dots , 0) : \theta_1,\dots, \theta_\delta \in [0,2\pi) \}.
\]
The condition that $(X,Y)$ is log Calabi--Yau ensures that the homotopy class of $\T$ does not depend on the choice of $p$. A similar description for the torus in Theorem \ref{thm:mainintro}(2), which may be found in (\ref{eq:degentori}).


One possible reason to be interested in Theorem \ref{thm:mainintro} is the P=W conjecture of de Cataldo, Hausel and Migliorini. For a detailed introduction (including proper definitions of the objects involved), see \cite{dchm}. For $G$ a reductive group, we denote by $M_\mathrm{H}$ the moduli space of (twisted) $G$-Higgs bundles on a curve $C$. This is also called the {\em Hitchin moduli space} associated to $C$ and $G$. The variety $M_\mathrm{H}$ is noncompact, hyperk\"ahler, and is of dimension $2d$ for some $d$. According to Hitchin \cite{hit}, there is an algebraic complete integrable Hamiltonian system on $M_\mathrm{H}$, which manifests as a proper map $h : M_\mathrm{H} \rightarrow \mathbb{A}^{d}$ called the Hitchin map. Associated to the same data, we may construct the moduli space of (twisted) $G$-representations of the fundamental group of $C$. This is the moduli space of $G$-local systems on $C \setminus p$ so that  so that monodromy around $p$ is conjugate to a fixed finite order automorphism. We will call this space the {\em Betti moduli space} and denote it $M_\B$. By the nonabelian Hodge correspondence, $M_\B$ and $M_\mathrm{H}$ are related by hyperk\"ahler rotation and therefore are diffeomorphic to one another.

Clearly, $M_\B$ and $M_\mathrm{H}$ have identical cohomology. However, since $M_\B$ and $M_\mathrm{H}$ are not generally algebraic deformations of one another, it is unclear whether there is any relationship between the various Hodge theoretic filtrations on $H^k(M_\B;\mathbb{Q})$ and those of $H^k(M_\mathrm{H};\mathbb{Q})$. The P=W conjecture roughly proposes that the weight filtration on $M_\B$ should match the perverse Leray filtration on $M_\mathrm{H}$ coming from the Hitchin map. The {\em perverse Leray filtration}  (denoted $P_\bullet$) is the filtration on the cohomology of $M_\mathrm{H}$ induced by the perverse truncations of $Rh_*{\mathbb{Q}}_{M_\mathrm{H}}$. A precise statement of the P=W conjecture is as follows.
\begin{conjecture}[de Cataldo, Hausel, Migliorini \cite{dchm}]\label{conj:P=W}
Letting $M_\mathrm{H}$ and $M_\B$ be as above, we have that 
\[
W_{2j-2}\HH^k(M_\B;\mathbb{Q}) = W_{2j-1}\HH^k(M_\B;\mathbb{Q}) = P_{j-1}\HH^k(M_\mathrm{H};\mathbb{Q})
\]
for all $j$ and $k$. 
\end{conjecture}
This conjecture has been proved in the case where $G = \mathrm{SL}_2(\mathbb{C}), \mathrm{GL}_2(\mathbb{C})$, and $\mathrm{PGL}_2(\mathbb{C})$ in \cite{dchm}. Even in these simple cases, the proof is formidable and requires the application of a number of difficult results. Theorem \ref{thm:mainintro}(1) offers some insight into the geometry of Conjecture \ref{conj:P=W} at the highest weight.

According to de Cataldo and Migliorini \cite{dcm} the perverse Leray filtration can be computed using the preimages of a collection of general hyperplanes $H_1,\dots, H_{d}$ on the Hitchin base $\mathbb{A}^{d}$ and their intersections. Particularly, if we let $F_h$ be a smooth fiber of $h$ (which is an abelian variety of dimension $d$), then
\begin{equation}\label{eq:plow}
P_{k-1} \HH^k(M_\mathrm{H};\mathbb{Q}) = \mathrm{ker}(\HH^k(M_\mathrm{H};\mathbb{Q}) \longrightarrow \HH^k(F_{h};\mathbb{Q})).
\end{equation}
If Conjecture \ref{conj:P=W} holds, then it implies, along with (\ref{eq:plow}), that there is a non-algebraic torus $F_h$ in $M_\B$ of real dimension $2d$ so that 
\[
W_{2k-1}\HH^k(M_\B;\mathbb{Q}) = \mathrm{ker}(\HH^k(M_\B;\mathbb{Q}) \longrightarrow \HH^k(F_h;\mathbb{Q})).
\]
Simpson \cite{simps} has conjectured that $M_\B$ admits a log Calabi--Yau compactification for which $\delta = \dim M_\B$. If this is true, then Theorem \ref{thm:mainintro}(1) provides a real $(\dim M_\B)$-dimensional torus with the same properties as the torus $F_h$ is expected to have. 

Even if $M_\B$ is log Calabi--Yau, it is not clear whether $\T$ and $F_h$ are homotopic to one another, however this is implied by a general conjecture of Auroux. Since $F_h$ is the fiber of a holomorphic Lagrangian torus fibration on $M_\mathrm{H}$ it is the fiber of a special Lagrangian torus fibration on $M_\B$. A conjecture of Auroux \cite[Conjecture 7.3]{auroux} implies that if $(X,Y)$ is a log Calabi--Yau pair, then the fiber of any special Lagrangian torus fibration on $X\setminus Y$ is homotopic to $\T$. Therefore the conjectures of Simpson and Auroux combined with Theorem \ref{thm:mainintro}(1) imply the P=W conjecture at the highest weight. Moreover, this seems to suggest that the P=W conjecture is true at the highest weight for any pair $M_\B$ and $M_\mathrm{H}$ satisfying the ``nonabelian Hodge package'' (\cite[\S 1]{boalch}) under the condition that $M_\B$ admits a log Calabi--Yau compactification.

In \cite{knps}, Katzarkov, Noll, Pandit, and Simpson have made a related conjecture called the ``geometric P=W conjecture''. If one chooses a compactification $\overline{M_{\B}}$ of $M_\B$, there is a differentiable torus fibration on a neighbourhood of $\overline{M_\B} \setminus M_\B$ in $M_\B$ whose generic fiber is $\T$ and whose base is a sphere of dimension $2d-1$ (see \cite{simps2} for a precise explanation). This fibration is constructed from $\overline{M_{\B}}$ and is called the {\em Betti Hitchin map}. Katzarkov, Noll, Pandit, and Simpson conjecture that the Betti Hitchin map can be identified with the restriction of the Hitchin map to a large sphere in $\mathbb{A}^d$. Theorem \ref{thm:mainintro}(1) provides a link between the geometric P=W conjecture and the P=W conjecture of de Cataldo, Hausel, and Migliorini. 

In the case where $\dim M_\B = 2$, the perverse Leray filtration has length 2, so Theorem \ref{thm:mainintro}(1) implies that the P=W conjecture follows from the geometric P=W conjecture. Similar ideas appear in work of Szab\'o \cite{szabo1,szabo2}. We will use this idea to prove the following ``P=W type'' theorem which, in some sense, generalizes a recent result of Zhang \cite{zh}.
\begin{theorem}[Theorem \ref{thm:symhk}, Proposition \ref{prop:complexstr}, Theorem \ref{thm:pw}]
Let $X$ be a rational surface and let $Y$ be a reduced, nodal anticanonical divisor in $X$. Then the following statements are true.
\begin{enumerate}
\item For some symplectic form $\omega$ on $X$, there is a Lagrangian torus fibration $g : X\setminus Y\rightarrow \Delta^\circ$ for some open subset $\Delta^\circ$ of $\mathbb{R}^2$.
\item There is a different (usually non-algebraic) complex structure on $X\setminus Y$ in which $g$ is a holomorphic elliptic fibration over a disc.
\item The perverse Leray filtration on $H^*(X\setminus Y;\mathbb{Q})$ with respect to $g$ has the property that $P_i = W_{2i+1} = W_{2i}$ for all $i$.
\end{enumerate}
\end{theorem}
Finally, the fact that Theorem \ref{thm:mainintro}(1) has a compact analogue in Theorem \ref{thm:mainintro}(2) suggests that the P=W conjecture has an analogue for compact hyperk\"ahler varieties. Indications that such an analogue might exist already appear in the recent work of Shen and Yin \cite{s-y} and in the much older work of Gross \cite{gross}. Assume that $M$ is compact and hyperk\"ahler, and that it admits a holomorphic Lagrangian torus fibration $\ell: M \rightarrow B$. Let $\beta$ be the pullback of an ample divisor on $B$. Let $q(\bullet,\bullet)$ be the Beauville--Bogomolov--Fujiki form on $H^2(M;\mathbb{Z})$. The class $\beta$ has the property that $q(\beta,\beta) = 0$, hence it defines a type III boundary point in the moduli space of K\"ahler deformations of $M$. As $M$ deforms towards this boundary point, one obtains a limit mixed Hodge structure which is equipped with a monodromy weight filtration. Let $M_{1}$ be a deformation of $M$ which is near the boundary point determined by $\beta$. Let $W_\bullet H^k(M_1;\mathbb{Q})$ be the monodromy weight filtration corresponding to the degeneration of $M_1$ at the boundary point determined by $\beta$.

\begin{conjecture}\label{conj:P=Wcompact}
Let notation be as above. Then there is a diffeomorphism between $M$ and $M_1$ so that perverse Leray filtration associated to $\ell$ and the weight filtration corresponding to the limit mixed Hodge structure of $(\mathscr{M},\pi)$ have the property that 
\[
P_iH^k(M;\mathbb{Q}) = W_{2i}H^k(M_1;\mathbb{Q}) = W_{2i+1}H^k(M_1;\mathbb{Q}) 
\]
for all $i$ and $k$.
\end{conjecture}
This conjecture seems very plausible to us. In fact one may deduce that
\[
\dim P_iH^k(M;\mathbb{Q}) = \dim W_{2i}H^k(M_1;\mathbb{Q}), \, \, \dim W_{2i+1}H^k(M_1;\mathbb{Q})=0
\] 
from work of Shen and Yin \cite{s-y} along with results of Soldatenkov \cite{sold}. The main result of \cite{s-y} is that if $M$ is a compact hyperk\"ahler variety which admits a Lagrangian torus fibration $\ell : M \rightarrow B$, then
\begin{equation}\label{eq:s-y}
\dim \Gr^P_iH^{i+j}(M) = \dim \Gr_F^{i}H^{i+j}(M) 
\end{equation}
for all $i$ and $j$. The dimension of $\Gr_F^iH^{i+j}(M;\mathbb{C})$ is a deformation invariant of $M$. As noted above, given $\ell$, there is a degeneration $(\mathscr{M},\pi)$ of hyperk\"ahler manifolds whose smooth fibers are diffeomorphic to $M$ and whose monodromy operator is related $\beta$. Let $H_\mathrm{lim}^{i+j}$ denote the limit mixed Hodge structure associated to the degeneration $(\mathscr{M}, \pi)$. In \cite{sold}, Soldatenkov proves that if a monodromy operator $N$ associated to a semistable degeneration $(\mathscr{M},\pi)$ of compact hyperk\"ahler manifolds has the property that $N^2|_{H^2(M)} \neq 0$, then the limit mixed Hodge structure of this degeneration is {\em Hodge--Tate}. This means that, for all $i$ and $j$,
\begin{equation}\label{eq:sold}
\dim \Gr^W_{2i} H^{i+j}_\mathrm{lim} = \dim \Gr_F^i H^{i+j}_\mathrm{lim}, \quad \dim  \Gr^W_{2i+1}H^{i+j}_\mathrm{lim} = 0.
\end{equation}
Combining (\ref{eq:s-y}) and (\ref{eq:sold}), we obtain the identity
\[
\dim \Gr^P_iH^{i+j}(M) = \dim \Gr^W_{2i} H^{i+j}_\mathrm{lim}
\]
hence Conjecture \ref{conj:P=Wcompact} holds on the numerical level. Here we have used the fact that $\dim \Gr_F^iH^{i+j}_\mathrm{lim} = \dim \Gr_F^i H^{i+j}(M;\mathbb{Q})$. Conjecture \ref{conj:P=Wcompact} says that this is obtained from an identification of filtrations. Theorem \ref{thm:mainintro}(2) allows us to prove this conjecture in dimension 2.
\begin{theorem}[Theorem \ref{thm:pwk3}]
Conjecture \ref{conj:P=Wcompact} is true for K3 surfaces.
\end{theorem}


\subsection*{Organization}
Sections \ref{sect:lcy} and \ref{sect:lmhs} are devoted to the proof of Theorem \ref{thm:mainintro}. Both sections follow the same approach. We first use topological data and work of El Zein and N\'emethi \cite{elz-n} (resp. Clemens \cite{clem}) to prove that for any snc pair $(X,Y)$ (resp. semistable degeneration) there is a collection of real tori $\T_1,\dots, \T_m$ from which one can compute the lowest weight piece of the mixed Hodge structure on the corresponding homology group. Then we apply results of Koll\'ar \cite{kollar} to prove the two parts of Theorem \ref{thm:mainintro} separately.

In Section \ref{sect:cons}, we prove P=W type results for certain surfaces. First, if $(X,Y)$ is log Calabi--Yau pair so that $Y$ has at least one node, we use work of Symington \cite{sym} and Gross, Hacking and Keel \cite{ghk} to show that there is an elliptic Lefschetz fibration $g : X \setminus Y \rightarrow \mathbb{R}^2$. Then we argue that this elliptic Lefschetz fibration admits a complex structure and we prove that the perverse Leray filtration associated to this holomorphic fibration is equal to the weight filtration in the original complex structure. We then state and prove Conjecture \ref{conj:P=Wcompact} in the case where $\dim M = 2$.


\subsection*{Acknowledgements}

I'd like to thank Ludmil Katzarkov, Tony Pantev, and Morgan Brown for stimulating conversations regarding the subject matter of this paper.


\section{The weight filtration for a snc log Calabi-Yau pair}\label{sect:lcy}

In this section, we prove Theorem \ref{thm:mainintro}(1). Throughout this paper, we will assume that the reader is familiar with the basic formalism of mixed Hodge structures. We will introduce facts as needed. For a formal introduction, the reader may consult the book of Peters and Steenbrink \cite{pet-st} or the original papers of Deligne \cite{deligne2,deligne}.

\subsection{Review of work of El Zein and N\'emethi}\label{sect:elzn}


According to Deligne \cite{deligne2}, if $U$ is any  quasiprojective variety over the complex numbers, then its cohomology groups $H^*(U;\mathbb{Q})$ come equipped with a mixed Hodge structure.
\begin{defn}
Let $V$ be a rational vector space, and let $W_\bullet$ be an ascending filtration and let $F^\bullet$ be a descending filtration on $V \otimes \mathbb{C}$. We say that $(V,W_\bullet,F^\bullet)$ is a {\em mixed Hodge structure} if the filtration on $\Gr^{W\otimes \mathbb{C}}_i = (W_i\otimes\mathbb{C})/(W_{i-1}\otimes \mathbb{C})$ induced by $F^\bullet$ is a pure Hodge structure. 
\end{defn}
In this paper, we will focus on the case where $U$ is smooth, in which case we may always take a compactification $X$ of $U$ so that $Y = X \setminus U$ is simple normal crossings. We also have a nondegenerate pairing  between $H^k(U;\mathbb{Q})$ and $H_k(U;\mathbb{Q})$ which induces a mixed Hodge structure on $H_k(U;\mathbb{Q})$ for any smooth quasiprojective variety. Our convention in this article is that $W_{-i}H_k(X \setminus Y ;\mathbb{Q})$ is the subset of $H_k(X\setminus Y;\mathbb{Q})$ composed of $k$-cycles $\xi$ so that 
\[
\int_\xi \omega = 0
\]
for all $\omega \in W_iH^k(X\setminus Y;\mathbb{Q})$. In \cite{elz-n}, El Zein and N\'emethi have identified homological cycles spanning each component of the weight filtration on a smooth noncompact variety via ``generalized Leray cycles''. We will recall the relevant parts of their work. 

We let $Y_1,\dots, Y_n$ be the irreducible components of $Y$, and let $Y_I = \cap_{i \in I} Y_i$. As a convention, we let $Y^0 = X$. We let $Y^i = \cup_{|I| = i} Y_I$, and we let $\mathbf{n}:\widetilde{Y}^i \rightarrow Y^i$ be the normalization of $Y^i$. An important point in \cite[Section 1.2]{elz-n} is that one can construct a tubular neighborhood of $Y$ within $X$. For $i \in \{1,\dots ,n\}$, we let $Z_i$ be the real oriented blow up of $Y_i$ in $X$. Let $\Pi_i : Z_i \rightarrow X$ be the real oriented blow up map. Then $Z$ is a manifold with boundary, and in fact this boundary (which is $\Pi_i^{-1}(Y_i)$) is homeomorphic to the normal $S^1$ bundle of $Y_i$ in $X$. We let $Z$ be the fiber product of the set of maps $\Pi_i$ over $X$. Then $Z$ is homeomorphic to the complement of a tubular neighborhood $V$ of $Y$ and there is a natural map $\Pi : Z \rightarrow X$. According to \cite{elz-n}, $Z$ is a deformation retract of $X \setminus Y$. We view $Z$ as a submanifold (with corners) of $X$. If $p$ is a point in $Y^i\setminus Y^{i+1}$, then $\Pi^{-1}(p)$ is a copy of $(S^1)^i$.

The goal of \cite{elz-n} is to produce a generalized Leray map, by which we mean a way to lift cycles from $Y$ to cycles in $X \setminus Y$, and to understand the relation between this Leray map and the weight filtration. More details may be found \cite[Section 2.19]{elz-n}.
\begin{defn}
We let $C_q^\pitchfork(\widetilde{Y}^k)$ denote rational $q$-cycles on $\widetilde{Y}^k$ which intersect the subset of $\widetilde{Y}^k$ corresponding to $\widetilde{Y}^{k+1}$ transversally. 
\end{defn}
\begin{defn}
Let $\xi$ be a $q$-cycle in $C_q^\pitchfork(\widetilde{Y}^k)$. If $\sigma_0$ is $\mathbf{n}(\xi) \setminus (Y^{k+1}\cap \mathbf{n}(\xi))$, then we define $L_k(\xi)$ to be the closure of $\Pi^{-1}(\sigma_0)$ in $Z$. 
\end{defn}

\begin{defn}
We define the double complex $A_{s,t}(X \setminus Y) = C^\pitchfork_{t+2s}(\widetilde{Y}^{-s})$ with $s\leq 0$ and $t+2s \geq 0$, where the differentials are given as
\[
\partial : C^\pitchfork_{t+2s}(\widetilde{Y}^{-s}) \longrightarrow  C^\pitchfork_{t-1+2s}(\widetilde{Y}^{-s}), \quad \cap : C^\pitchfork_{t+2s}(\widetilde{Y}^{-s}) \longrightarrow  C^\pitchfork_{t+2s-2}(\widetilde{Y}^{-s+1}).
\]
Here, $\partial$ denotes the standard boundary map, and $\cap$ denotes transversal intersection. Then $D = \partial + \cap$ is the differential of the total complex $\mathrm{Tot}_\bullet(A_{s,t}(X \setminus Y))$.
\end{defn}
The Leray maps $L_q :C_t^\pitchfork(\widetilde{Y}_q) \rightarrow C_{t+q}(\partial Z)$ extend to a collection of maps $L:\mathrm{Tot}_\bullet(A_{s,t}(X\setminus Y)) \rightarrow C_\bullet(\partial Z)$, which (by \cite[Corollary 2.21]{elz-n}) is a morphism of complexes. Combining this with the pushforward from $\partial Z$ to $X\setminus Y$, we induce a map from the hypercohomology of $\mathrm{Tot}_\bullet(A_{s,t}(X\setminus Y))$ to the homology of $X\setminus Y$. 
\begin{proposition}[{\cite[Proposition 3.8]{elz-n}}]\label{prop:ezn}
The hypercohomology of $\mathrm{Tot}_\bullet(A_{s,t}(X \setminus Y))$ is isomorphic to $\HH_*(X\setminus Y;\mathbb{Q})$. The filtration on $\HH_*(X\setminus Y;\mathbb{Q})$ induced by truncation on $\mathrm{Tot}_\bullet(A_{s,t}(X\setminus Y))$ is the weight filtration on homology as defined above.
\end{proposition}

For a general closed cycle $c_{s,t} \in A_{s,t}(X \setminus Y)$, \cite[\S 4]{elz-n} explains how to obtain a homologous closed cycle $c^\infty_{s,t} = c_{s,t} + c_{s-1,t+1}  + \dots + c_{0,s+t}$ so that the cycles $L c_{s,t}^\infty$ generate $W_{-t}\HH_{s+t}(X \setminus Y)$ (\cite[\S 5.10]{elz-n}). For our purposes, the exact construction of $c^\infty_{s,t}$ is irrelevant, since we are interested in the case where $s = 0$, hence $c^\infty_{0,t} = c_{0,t}$. Proposition \ref{prop:ezn} then says that the cycles $L c_{0,t}$ generate $W_{-2k}H_k(X\setminus Y;\mathbb{Q})$. Let us now give an explicit description of what these cycles look like.
\begin{example}
If $[p] \in C_0^\pitchfork(\widetilde{Y}^{k})$, then we may assume that $p$ is a point away from ${Y}^{k+1}$. Then the cycle $L_k([p])$ can be expressed as follows. Choose complex coordinates $z_1,\dots, z_d$ in a neighbourhood of $p$ so that $Y = Z(z_1\dots z_k)$ and $p = (0,0,\dots, 0)$. Then the cycle $L_k([p])$ is homologous to the submanifold 
\[
\{(\exp(\mathtt{i} \theta_1), \dots, \exp(\mathtt{i} \theta_{k}), 0 , \dots, 0): \theta_i \in [0,2\pi)\}
\]
\end{example} 
\begin{defn}
Let $p$ be a point in $Y_I \setminus (Y_I \cap Y^{i+1})$. We will denote by $\T_{I,p}$ the $|I|$-torus $L_{|I|}([\mathrm{pt}])$ in $X\setminus Y$.  For any two such points, $p_1$ and $p_1$, $\T_{I,p_1}$ and $\T_{I,p_2}$ are homotopic. We will use the notation $\T_I$ to refer to any member of this homotopy class.
\end{defn}
The following statement then follows directly from this discussion.
\begin{proposition}[{\cite[\S 5.10]{elz-n}}]\label{prop:elzne}
For every $k$, the group $W_{-2k}\HH_k(X \setminus Y;\mathbb{Q})$ is generated by the classes $[\T_I]$ as $I$ ranges over all $I \subset \{1,\dots, n\}$ so that $|I| = k$ and $Y_I \neq \emptyset$.
\end{proposition}


\subsection{Consolidating tori}\label{sect:consolidation}
Our goal now is to show that the tori $\T_I$ can be drawn from a common source. In other words, we will show that there are several homotopy classes of tori coming from the ``deepest'' strata of $Y$ so that the image of their pushforwards in homology contain the homology classes of all tori $\T_I$.
\begin{proposition}\label{prop:restr}
Let $I \subseteq J \subseteq \{1,\dots, n\}$.
\begin{enumerate}
\item $\T_{I}$ is homotopic to a sub-torus of $\T_J$.
\item As $I$ ranges over all subsets of $J$ of cardinality $i$, the tori $\T_{I}$ form a generating set for $\HH_{|I|}(\T_J;\mathbb{Z})$.
\end{enumerate} 
\end{proposition}
\begin{proof}
Let us assume that $J = \{1,\dots, |J|\}$, which we can always do by reordering $Y_1,\dots, Y_n$. Let $p \in Y_J \setminus (Y_J \cap  Y^{|J|+1})$, and let $q \in p \in Y_I \setminus (Y_I \cap  Y^{|I|+1})$. First we note that the point we choose to construct $\T_J$ and $\T_{I}$ is irrelevant, since we may deform any fiber over a point in a connected component of $\widetilde{Y}_\ell$ to any other point in the same connected component. Therefore, we can assume that $p$ and $q$ are contained in a polydisc $B$ so that in these coordinates, $Y$ is written as $z_1 \dots z_{|J|} = 0$, and $Y_j \cap B$ is given by $z_i = 0$ for all $j \in J$. We may assume that $ p = (0,\dots,0)$, and that $q$ has coordinate $(\zeta_1,\dots, \zeta_d)$ where $\zeta_j = 0$ if $j \in I$, $\zeta_j \neq 0$ if $j \in J \setminus I$ and $\zeta_j$ is arbitrary if $j \notin J$. Then $\T_J$ is homotopic to
\[
\{ (\exp(\mathtt{i} \theta_1), \dots , \exp(\mathtt{i} \theta_{|J|}), 0, \dots, 0), \,\, \theta_1,\dots, \theta_{|J|} \in S^1 \}
\]
and $\T_I$ is homotopic to the torus whose coordinates are
\[
\begin{cases} \exp({\tt i} \theta_j) &  \mathrm{if } \quad j \in I \\
\zeta_i & \quad \mathrm{otherwise}.
\end{cases}
\]
Here, $\theta_j$ are constants in $[0, 2\pi)$. Then (1) can be seen immediately by letting $\zeta_j$ go to 1 for all $j \notin I$. From the K\"unneth theorem, the sub-tori of $\T_J$ obtained by letting some subset of $\{\theta_j\}_{j \in J}$ of size $|J|-|I|$ be 1 will span $H_{|I|}(\T_J;\mathbb{Z})$. Therefore, (2) follows.
\end{proof}

This will allow us to prove the main theorem of this section. 
\begin{defn}
We say that a stratum $Y_J$ is {\em minimal} if there is no $J'$ so that $J \subsetneq J' \subset \{1,\dots, n\}$ and $Y_{J'} \neq \emptyset$. If $Y_{J}$ is minimal, then we say that the corresponding torus $\T_J$ is {\em profound}. Let $\mathrm{Prf}(X,Y)$ be a set consisting of one profound torus $\T_J$ for each minimal stratum $Y_J$ of $Y$.
\end{defn}
By definition, each $Y_J$ contains at least one minimal stratum, hence by Proposition \ref{prop:restr}, each torus $\T_K$ is homotopic to a subtorus of a profound torus.
\begin{theorem}\label{thm:prof}
For each $k$,
\[
W_{-2k}\HH_k(X\setminus Y;\mathbb{Q}) = \mathrm{im}\left( \bigoplus_{\T_J \in \mathrm{Prf}(X,Y)} \HH_k(\T_J;\mathbb{Q}) \longrightarrow \HH_k(X \setminus Y;\mathbb{Q})\right)
\]
or dually,
\[
W_{2k-1}\HH^k(X \setminus Y;\mathbb{Q}) = \mathrm{ker}\left( \HH^k(X \setminus Y;\mathbb{Q}) \longrightarrow \bigoplus_{\T_J \in \mathrm{Prf}(X,Y)} \HH^k(\T_J;\mathbb{Q})\right).
\]
\end{theorem}
\begin{proof}
As $K$ ranges over all subsets of $\{1,\dots, n\}$ of cardinality $k$ so that $Y_K$ is nonempty, the tori $\T_K$ span $W_{-2k}\HH_k(X \setminus Y;\mathbb{Q})$ by Proposition \ref{prop:elzne}. If $\T_K$ is profound, then the homology class $[\T_K]$ is  the image of the pushforward $\HH_k(\T_K;\mathbb{Q}) \rightarrow \HH_k(X \setminus Y;\mathbb{Q})$. If $\T_K$ is not profound, then by Proposition \ref{prop:restr}(1), there is a deformation of $\T_K$ to a subtorus of a profound torus $\T_L$ with $K \subset L$, hence the homology class $[\T_K]$ is in the image of the pushforward map $\HH_k(\T_L;\mathbb{Q}) \rightarrow \HH_k(X\setminus Y;\mathbb{Q})$. Thus $W_{-2k}\HH_k(X \setminus Y;\mathbb{Q})$ is contained in the image of $\HH_k(\T_L;\mathbb{Q})$ as $\T_L$ ranges over all elements of $\mathrm{Prf}(X,Y)$.

On the other hand, if $\T_L$ is a profound torus, then $Y_L$ contains all $Y_I$ for $I \subseteq L$. Therefore, by Proposition \ref{prop:restr}(2), the image of $\HH_k(\T_L;\mathbb{Q})$ is spanned by classes which are deformations of $[\T_I]$ for $I \subseteq L$. Therefore, the pushforward map $\HH_k(\T_L;\mathbb{Q})$ has image spanned by tori $[\T_I]$ for $I \subseteq L$. Therefore $W_{-k}\HH_k(X \setminus Y)$ contains the image of the pushforward of $\HH_k(\T_L;\mathbb{Q})$ as $\T_L$ ranges over elements of $\mathrm{Prf}(X,Y)$.
\end{proof}


\subsection{Specialization to the Calabi-Yau case}\label{sect:mon}

In this section, we assume that $(X,Y)$ is a log Calabi-Yau pair, or in other words, that $K_X + Y$ is trivial. The geometry of such divisors is well-understood through work of Koll\'ar \cite{kollar}. We will show that in this case, all profound tori are homotopic to one another. Our main tool in this section is the following statement.
\begin{lemma}\label{lemma:deftori1}
Suppose that $p_1 \in Y_{J_1}\setminus (Y_{J_2} \cap Y^{|J_2|+1})$, $p_2 \in Y_{J_2} \setminus (Y_{J_1} \cap Y^{|J_1|+1})$ with $|J_1| = |J_2| = j$, and that there is a rational curve $C$ in $Y^{j-1}$ so that $C \cap Y^j$ is precisely the pair of points $p_1$ and $p_2$. Then the tori $\T_{J_1}$ and $T_{J_2}$ are homotopic. 
\end{lemma}
\begin{proof}
Recall from the discussion in Section \ref{sect:elzn} there is an  embedding of $Z$, the real oriented blow up of $X$ in $Y$, into $X \setminus Y$. We would like to show that $\T_{J_1}$ and $\T_{J_2}$ are homotopic in $Z$, hence they are homotopic in $X \setminus Y$. We use the notation $\Pi: Z \rightarrow X$ to denote the real oriented blow up map. By the assumption that $C$ intersects $Y^{j}$ only in $p_1$ and $p_2$, functoriality of real oriented blow up shows that $\Pi^{-1}(C)$ is identified with an $(S^1)^{j-1}$ fibration over the real oriented blowup of $C$ at $p_1$ and $p_2$. The real oriented blow up of $C$ at $p_1$ and $p_2$ is topologically, the cylinder $S^1 \times [0,1]$. The tori $\Pi^{-1}(p_1)$ and $\Pi^{-1}(p_2)$ are the preimages of $S^1 \times \{0\}$ and $S^1 \times \{1\}$, therefore, they are homotopic in $Z$.
\end{proof}

In \cite[Theorem 10]{kollar}, Koll\'ar prove a general result relating to the $\mathbb{P}^1$ connectedness of log canonical centers of dlt log Calabi--Yau pairs. We will state Koll\'ar's result in the generality needed for the situation at hand. Later on (Section \ref{sect:cycase}) we will describe the consequences of the same result in the case of certain degenerations of Calabi--Yau varieties.

\begin{theorem}[{\cite[Theorem 10]{kollar}}]\label{lemma:kollar1}
Suppose that $(X,Y)$ is log Calabi--Yau, then if $J_1$ and $J_2$ are subsets of $\{1, \dots, n\}$ so that $Y_{J_1}$ and $Y_{J_2}$ are minimal. Then $|J_1| = |J_2| = j$ for some $j$, and there are points $p_1 \in Y_{J_1}$ and $p_2 \in Y_{J_2}$ which are connected by a chain of rational curves $C_1, \dots, C_m$ in $Y^{j-1}$, each having the property that $C_i$ intersects $Y^j$ transversally in a distinct pair of points.
\end{theorem}

We are now equipped to prove the main theorem of this section.

\begin{theorem}\label{thm:main}
Let $(X,Y)$ be an snc log Calabi--Yau pair. Let $\T = \T_J$ for any maximal stratum $Y_J$. Then
\[
W_{2k-1}\HH^k(X \setminus Y;\mathbb{Q}) \cong \mathrm{ker}( \HH^k(X \setminus Y;\mathbb{Q}) \longrightarrow \HH^k(\T;\mathbb{Q})).
\]
\end{theorem}
\begin{proof}
According to Theorem \ref{thm:prof}, we know that $W_{2k-1}H^k(X\setminus Y;\mathbb{Q})$ is the kernel of the restriction to the direct sum of $H^k(\T_I;\mathbb{Q})$ as $I$ ranges over all elements of  $\mathrm{Prf}(X, Y)$. Combining Lemma \ref{lemma:deftori1} and Theorem \ref{lemma:kollar1}, we see that all $\T_I \in \mathrm{Prf}(X,Y)$ are homotopic, hence for any individual $\T \in \mathrm{Prf}(X,Y)$ the kernel is the same as the kernel of the restriction to the direct sum. Therefore the result follows.
\end{proof}
\begin{remark}
Note that if
\begin{enumerate}[($*$)]
\item for any minimal strata $Y_{J_1}$ and $Y_{J_2}$, $|J_1| = |J_2| = j$ for some $j$, and there are points $p_1 \in Y_{J_1}$ and $p_2 \in Y_{J_2}$ which are connected by a chain of rational curves $C_1, \dots, C_m$ in $Y^{j-1}$, each having the property that $C_i$ intersects $Y^j$ transversally in a distinct pair of points,
\end{enumerate}
then the conclusions of Theorem \ref{thm:main} also hold. Theorem \ref{lemma:kollar1} shows that ($*$) is satisfied whenever $(X,Y)$ is a log Calabi--Yau pair. If the P=W conjecture is true, then one expects that there is a snc compactification of $M_\B$ satisfying at least ($*$).
\end{remark}
\begin{corollary}\label{cor:bound}
Let $(X,Y)$ be a log Calabi--Yau pair and let $\delta$ be the codimension of any maximal stratum of $Y$. Then 
\[
\dim \Gr^W_{2k}\HH^k(X \setminus Y;\mathbb{Q})  \leq {\delta \choose k}.
\]
\end{corollary}
\begin{remark}
In the compact case one can deduce a result similar to Corollary \ref{cor:bound} from the Beauville--Bogomolov decomposition theorem \cite{beauville}. This theorem says that any smooth, compact Calabi--Yau manifold (i.e. a manifold with trivial canonical bundle) has an unramified covering map from a product of varieties 
\[
X_1\times \dots \times X_m \times I_1\times \dots \times I_\ell \times A_1 \times \dots \times A_n
\]
where 
\begin{enumerate}[\quad $\bullet$]
\item $X_i$ are simply connected Calabi--Yau manifolds so that $h^{m,0}(X_i) = 0$ if $i \neq 0, \dim X_i$, and $\dim h^{\dim X,0}(X) = 1$,
\item $I_j$ are irreducible holomorphic symplectic, hence $h^{2n,0}(I_i) =1$ for $0 \leq n \leq \dim I_i/2$, and $h^{m,0}(I_i) = 0$ otherwise, 
\item $A_i$ are abelian varieties. 
\end{enumerate}
Therefore it follows that for any Calabi--Yau manifold $V$, 
\[
\dim \Gr_F^kH^k(V) \leq {\dim V \choose k }.
\]
This raises the question as to whether an analogue of the Beauville--Bogomolov decomposition theorem holds for log Calabi--Yau pairs.
\end{remark}


\section{The weight filtration for a Calabi--Yau degeneration}\label{sect:lmhs}

In this section, we will prove Theorem \ref{thm:mainintro}(2). Our approach is almost identical to our approach to proving Theorem \ref{thm:mainintro}(1), however instead of using work of El Zein and N\'emethi, it will be necessary to modify results of Clemens \cite{clem}.

\subsection{The monodromy weight filtration}

In this section, we will describe the lowest piece of the weight filtration on the homology of a semistable degeneration in concrete terms using the Clemens contraction map.

\begin{defn}
Let $\mathscr{X}$  be a K\"ahler manifold and let $\pi:\mathscr{X} \rightarrow \Delta$ be a proper morphism of relative dimension $d$ where $\Delta$ denotes the unit disc in $\mathbb{C}$ centered at 0. Let $X_t = \pi^{-1}(t)$. We say that $(\mathscr{X},\pi)$ is a {\em semistable degeneration} if $X_t$ is smooth and projective whenever $t \neq 0$, $X_0$ is simple normal crossings, and $\pi$ vanishes to order 1 along each component of $X_0$.
\end{defn}
Associated to any semistable degeneration there is a {\em limit mixed Hodge structure} on the cohomology of $H^*(X_1;\mathbb{Q})$ (see e.g. \cite{schmid,steen,pet-st} for details). In keeping with the philosophy of this paper, we will ignore the Hodge filtration of this mixed Hodge structure and focus only on the weight filtration. The weight filtration of the limit mixed Hodge structure may be identified with the {\em monodromy weight filtration} on $H^*(X_1;\mathbb{Q})$. Let $T_k : H^k(X_1;\mathbb{Q}) \rightarrow H^k(X_1;\mathbb{Q})$ be the monodromy operator associated to a small counterclockwise loop going around $0 \in \Delta$. By the assumption that $(\mathscr{X},\pi)$ is semistable, it follows that $T_k$ is unipotent, hence $N_k = \log T_k$ is nilpotent. As an important remark, $N_k^{k+1} = 0$ for all $k$, and if $k > d$, $N_k^{2d-k+1} = 0$ \cite{clem}.
\begin{defn}\label{defn:mwf}
The {\em monodromy weight filtration} on $H^k(X_1;\mathbb{Q})$ associated to $T_k$ is the unique increasing filtration so that $N_k(M_j) \subseteq M_{j-2}$ for all $j$, and so that the induced map
\[
N_k^\ell : \Gr^M_{k+\ell}H^i(X_1;\mathbb{Q}) \longrightarrow \Gr^M_{k-\ell}H^k(X_1;\mathbb{Q})
\]
is an isomorphism for all $k$ and $\ell$.
\end{defn}
There  is a precise way of describing $M_\bullet$ (e.g. \cite[pp. 76]{cox-katz}), but we only need the fact that if $N_k^{\ell + 1} = 0$ then
\begin{equation}\label{eq:monfilt}
M_{k-\ell} H^k(X_1;\mathbb{Q}) = \mathrm{im}(N_k^\ell), \quad M_{k+\ell-1} H^k(X_1;\mathbb{Q}) = \ker(N_k^\ell).
\end{equation} 

\begin{proposition}
Let $k\leq d$. Then $M_{2k-1}H^k(X_1;\mathbb{Q})$ is the orthogonal complement of $M_{2d-2k}H^{2d-k}(X_1;\mathbb{Q})$.
\end{proposition}
\begin{proof}
First, we note that $N_k^{k+1} = 0$ and $N_{2d-k}^{k+1} = 0$, so by (\ref{eq:monfilt}), we have that 
\[
M_{2k-1}H^{k}(X_1;\mathbb{Q}) = \mathrm{im}(N_{k}^k), \quad M_{2d-2k-1}H^{2d-k}(X_1;\mathbb{Q}) = \ker(N_{2d-k}^k).
\]
Monodromy preserves the intersection pairing, so that $\langle T_k(\eta), T_{2d-k}(\zeta)\rangle = \langle \eta, \zeta \rangle$ for any $\eta \in H^k(X_1;\mathbb{Q})$ and $\zeta \in H^{2d-k}(X_1;\mathbb{Q})$. Therefore,
\[
\langle N_k\eta, \zeta \rangle \pm \langle \eta, N_{2d-k}\zeta\rangle = 0
\]
and thus, for any $\ell$, we also have
\begin{equation}\label{eq:pairingpm}
\langle N_k^\ell\eta, \zeta \rangle  = \pm \langle \eta, N^\ell_{2d-k}\zeta\rangle.
\end{equation}
By the nondegeneracy of the pairing $\langle \bullet, \bullet \rangle$, 
\[
\langle \eta, N^k_{2d-k}\zeta \rangle = 0 
\]
for all $\eta \in H^{k}(X_1;\mathbb{Q})$ if and only if $N_{2d-k}^k\zeta = 0$, in other words, if and only if $\zeta \in M_{2d-2k-1}H^{2d-k}(X_1;\mathbb{Q})$. By (\ref{eq:pairingpm}), this is true if and only if
\[
\langle N_k^k\eta, \zeta \rangle = 0
\]
for all $\eta$. In other words, if and only if $\zeta$ is in the orthogonal complement of $\mathrm{im}(N_k^k) = M_{2k-1}H^i(X_1;\mathbb{Q})$. Therefore, $\zeta \in M_{2d-2k}H^{2d-k}(X_1;\mathbb{Q})$ if and only if it is in the orthogonal complement of $M_{2k-1}H^k(X_1;\mathbb{Q})$. 

A nearly identical argument shows that $\eta \in M_{2k-1}H^{k}(X_1;\mathbb{Q})$ if and only if it is in the orthogonal complement of $M_{2d-2k}H^{2d-k}(X_1;\mathbb{Q})$. This completes the proof.
\end{proof}
Duality then identifies $H_k(X_1;\mathbb{Q})$ with $H^{2d-k}(X_1;\mathbb{Q})$ along with their monodromy actions. Thus we have the following result.
\begin{corollary}
Let $S_k$ denote the logarithm of the monodromy on $H_k(X_1;\mathbb{Q})$. Then, under the natural pairing, $M_{2k-1}H^k(X_1;\mathbb{Q})$ is the orthogonal complement of the image of $S_k^k$.
\end{corollary}

Our goal is to compute $M_{2k-1}H^k(X_1;\mathbb{Q})$ for a semistable degeneration. Thus we must compute the image of $S_k^k$ for all $k \leq d$. This is now a straightforward task, thanks to classical results of Clemens \cite{clem}.

The main tool involved in this computation is the {\em Clemens contraction map}. Let $(\mathscr{X},\pi)$ be a normal crossings compactification. Then Clemens constructs a deformation retract from $\mathscr{X}$ to $X_0$. The induced map $X_1 = \pi^{-1}(1) \rightarrow X_0$ is called the {\em Clemens contraction map}, and will be denoted by $r$. If $X_0$ is the union of divisors $A_1,\dots, A_n$, and for each $I \subset \{1,\dots, n\}$ we define $A_I = \cap_{i \in I}$ and ${A}^{i} = \cup_{|I| = i} A_I$. We will let $\widetilde{A}^i$ be the normalization of $A^i$. The Clemens contraction map has the property that the preimage of any point in ${A}^{i} \setminus {A}^{i+1}$ under $r$ is a torus of dimension $i -1$. As before, let us use the notation $\T_{I,p}$ to denote a torus in $X_1$ which is the preimage of a point $p$ in $A_I \setminus (A^{|I|+1} \cap A_I)$. In local coordinates, this can be described explicitly. If $p \in A^i \setminus A^{i+1}$, then we may choose local coordinates $z_1,\dots, z_d$ of $\mathscr{X}$ centered at $p$ so that $\pi = z_1\cdots z_i$ in this neighbourhood. Then
\begin{equation}\label{eq:degentori}
\T_{I,p} = \left\{ (\exp(\mathtt{i}\theta_1),\dots, \exp(\mathtt{i}\theta_i),0,\dots,0) : \theta_i \in S^1, \sum_{j=1}^i \theta_j = 0\right\}.
\end{equation}
The homotopy class of $\T_{I,p}$ in $X_1$ does not depend on the point $p$ in $A_I \setminus (A^{|I|+1} \cap A_I)$ or the local coordinates that we chose. We let $\T_I$ denote a member of this homotopy class. The next result is analogous to Proposition \ref{prop:elzne}.
\begin{theorem}[Clemens, \cite{clem}]
For all $k \leq d$, the image of $S_k^k$ in $H_k(X_1;\mathbb{Q})$ is spanned by the tori $\{\T_I : |I| = k+1\}$.
\end{theorem}
\begin{proof}[Sketch of proof]

For the sake of consistency with \cite{clem}, we use the index $q$ throughout the proof instead of $k$. Let us first remark that $S_q^q = (\log T_q)^q = (T_q - \mathrm{id})^q$. Therefore, it is enough to compute the image of $(T_q-\mathrm{id})^q$, which is what Clemens' work allows us to do.

For each torus $\T_I$, Clemens constructs a simplex $\sigma_I$ near $\T_I$ so that $(T_q-\mathrm{id})^q(\sigma_I)$ is a multiple of $\T_I$ (\cite[Formulae 3.3]{clem}), which implies that for each $I$ so that $|I| = q+1$, the torus $\T_I$ is in the image of $(T_q - \mathrm{id})^q$.

We now must check that the tori $\T_I$ span the image of $(T_q -\mathrm{id})^q$. We follow the beautiful (if somewhat arcane) proof of \cite[Theorem 4.4]{clem} closely, adapting the notation therein. The idea is very explicit. We choose a $q$-cycle $\alpha$ in $X_1$, then lift $\alpha$ to a homology group where it is homologous to a linear combination of cycles which look like cycles like $\sigma_I$ or things which are in the kernel of $(T_q - \mathrm{id})^q$. Close analysis of the cycles involved will give us our result. We note that \cite[Theorem 4.4]{clem} can be cited directly to prove our result in the case where $q = d$. The argument given below shows that in the special case of $(T_q-\mathrm{id})^q$, \cite[Theorem 4.4]{clem} can be strengthened.

We will let $C(I)$ be the subset of $A_I$ made up of points sufficiently far from $A_I \cap (\cup_{|I'|= |I|+ 1} A_{I'})$. Then we choose a cellular decomposition of all subsets $C(I)$ whose $p$-skeleton (which we denote $C(I)_p$) satisfies the assumptions of \cite[Lemma 4.6(2)]{clem}. Essentially, this means that the cellular decomposition is such that there is a local analytic chart containing $\sigma_I$ in which the function $\pi$ may be written as $z_1\dots z_{|I|} = 0$. We let\footnote{The reader is cautioned that this $Y_q$ is distinct from the objects $Y_I$ appearing in Section \ref{sect:elzn}. This notation is used to allow our work to be compared to \cite{clem} easily.}
\[
Y_q = \bigcup_{\substack{I \in \{1,\dots, k\} \\ p + |I| = q}} r^{-1}(C(I)_p)
\]
Then \cite[Lemma 4.7]{clem} says that the map $H_{q}(Y_{q+1};\mathbb{Q}) \rightarrow H_q(X_1;\mathbb{Q})$ is surjective, and \cite[Lemma 4.8]{clem} says that $H_{q}(Y_{q+1};\mathbb{Q}) \rightarrow H_{q}(Y_{q+1}, Y_{q};\mathbb{Q})$ is injective. Thus $H_{q}(X_1;\mathbb{Q})$ is a subquotient of $H_{q}(Y_{q+1},Y_{q};\mathbb{Q})$, so if we can compute the image of the map $(T_q - \mathrm{id})^q$ for any class in $H_{q}(Y_{q+1},Y_{q};\mathbb{Q})$, we will obtain our result. The choice of skeleton of $C(I)$ that we have made allows us to decompose $H_{q}(Y_{q+1},Y_{q};\mathbb{Q})$ in a nice way (c.f. the argument in \cite[pp. 103]{clem}), in particular, it is homologous in $H_q(Y_{q+1},Y_q;\mathbb{Q})$ to a sum of cycles $a \omega + \omega'$ for some constant $a$ where $(T_q - \mathrm{id})^q(\omega')$ is homologous to 0 in $H_q(Y_{q+1},Y_q;\mathbb{Q})$ (hence, also homologous to 0 in $H_q(X_0;\mathbb{Q})$), and $\omega$ is homologous to $\sigma_I$. Therefore $(T_q-\mathrm{id})^q\omega$ is homologous to a constant times the class of the torus $\T_I$. 
\end{proof}
\begin{remark}
The proof of \cite[Theorem 4.4]{clem} computes the image of $(T_{p+q} - \mathrm{id})^q$, in which case, the class $\omega$ mentioned in the proof above is expressed as a product $\sigma_I$ and a relative $p$-homology class. Then the image of $(T_{p+q} -\mathrm{id})^q$ is again a relative homology class. Our assumption that $p=0$ implies that the image of $(T_q-\mathrm{id})^q$ can be expressed as an absolute homology class, since any $p$-homology class is a point.
\end{remark}

Therefore, we may adapt the proof of Theorem \ref{thm:prof} in this case.
\begin{defn}
Let $(\mathscr{X},\pi)$ be a semistable degeneration so that $X_0$ has components $A_1,\dots, A_n$. Assume that $I \subseteq \{1,\dots, n\}$ is a maximal subset so that $A_I$ is nonempty. Then we say that $\T_I$ is a {\em profound torus}. We let $\mathrm{Prf}(\mathscr{X},\pi)$ be a set consisting of one profound torus $\T_I$ corresponding to a point in each maximal stratum $A_I$.
\end{defn}
\begin{theorem}
For each $k$,
\[
M_{-2k}H_k(X_1;\mathbb{Q}) = \mathrm{im}\left( \bigoplus_{\T_J \in \mathrm{Prf}(\mathscr{X},\pi)} H_k(\T_J;\mathbb{Q}) \longrightarrow H_k(X_1;\mathbb{Q})\right)
\]
or dually,
\[
M_{2k-1}H^k(X_1;\mathbb{Q}) = \mathrm{ker}\left( H^k(X_1;\mathbb{Q}) \longrightarrow \bigoplus_{\T_J \in \mathrm{Prf}(\mathscr{X},\pi)} H^k(\T_J;\mathbb{Q})  \right).
\]
\end{theorem}
\begin{proof}
Follow the proof of Theorem \ref{thm:prof}.
\end{proof}


\subsection{Specialization to the Calabi--Yau case}\label{sect:cycase}
Now we may specialize the results in the previous section to the case where $(\mathscr{X},\pi)$ is a Calabi--Yau degeneration.
\begin{defn}
We say that a semistable degeneration $(\mathscr{X},\pi)$ is {\em Calabi--Yau} if the canonical bundle $K_{\mathscr{X}}$ is trivial.
\end{defn}
We would like to show that all profound tori in a Calabi--Yau degeneration are homotopic to one another
\begin{lemma}\label{lemma:deftori}
Suppose that $p_1 \in A_{J_1}\setminus (A_{J_2} \cap A^{|J_2|+1})$ and $A_{J_2} \setminus (A_{J_1} \cap A^{|J_1|+1})$ with $|J_1| = |J_2| = j$, and that there is a rational curve $C$ in $A^{j-1}$ so that $C$ intersects $A^j$ transversally and that $C \cap A^j$ is precisely the pair of points $p_1$ and $p_2$. Then the tori $\T_{J_1}$ and $\T_{J_2}$ are homotopic to one another. 
\end{lemma}
\begin{proof}
This proof roughly follows the proof of Lemma \ref{lemma:deftori}. The connection to real oriented blow up in this case is not clear from the exposition above, but is made clear in work of Kawamata and Namikawa \cite[Proof of Lemma 4.1]{kn}. Kawamata and Namikawa show that the Clemens contraction map may be understood in terms of the real oriented blow up of $\mathscr{X}$ in $X_0$, which we denote $\mathrm{Blo}_{X_0}(\mathscr{X})$. The map $\pi : \mathscr{X} \rightarrow \Delta$ induces a map $\pi_0: \mathrm{Blo}_{X_0}(\mathscr{X}) \rightarrow \mathrm{Blo}_0(\Delta)$, where $\mathrm{Blo}_0(\Delta)$ is the real oriented blow up of $\Delta$ at $0$. Analytically, $\mathrm{Blo}_0(\Delta)$ is simply $\Delta \setminus D$ where $D$ is a small open disc containing 0, and the map $\pi_0$ is equivalent to $\pi$ restricted to $\pi^{-1}(\Delta \setminus D)$. The fiber $X'$ in $\mathrm{Blo}_{X_0}(\mathscr{X})$ over a point in the boundary of $\Delta \setminus D$ is a closed subset of $\partial \mathrm{Blo}_{X_0}(\mathscr{X})$, and the real oriented blow up map $\Pi : \mathrm{Blo}_{X_0}(\mathscr{X}) \rightarrow \mathscr{X}$ induces a map $X' \rightarrow X_0$ which is, topologically, the Clemens contraction map.

Let $p_1$ and $p_2$ be as above and let $(S^1)^i$ denote a real $i$-torus. Then the preimage of $C$ in $\mathrm{Blo}_{X_0}(\mathscr{X})$ is as in Lemma \ref{lemma:deftori}: the preimage of $C$ under $b$ is diffeomorphic to $(S^1)^{j} \times [0,1]$ and the preimages of of $p_1$ and $p_2$ can be identified with the tori $(S^1)^{j} \times \{0\}$ and $(S^1)^{j} \times \{1\}$ of dimension $j$ which we may denote $\T_{J_1}'$ and $\T_{J_2}'$ respectively. The map $\Pi^{-1}(C)$ to $S^1$ projects onto one of the $S^1$ factors of $(S^1)^j \times [0,1]$, hence the fiber over a fixed point in $S^1$ is diffeomorphic to $(S^1)^{j-1} \times [0,1]$. Then $\T_{J_1}$ and $\T_{J_2}$ may be identified with $(S^1)^{j-1} \times \{0\}$ and $(S^1)^{j-1} \times \{1\}$ respectively in $\Pi^{-1}(C)$. Therefore they are homotopic to one another.
\end{proof}
We may now specialize \cite[Theorem 10]{kollar} to the situation at hand to understand the geometry of the central fiber of a Calabi--Yau degeneration. This will allow us to relate the profound tori of a Calabi--Yau degeneration to one another.
\begin{theorem}[{Koll\'ar \cite[Theorem 10]{kollar}}]\label{lemma:kollar}
Suppose that $(\mathscr{X},\pi)$ is a Calabi-Yau degeneration and that $\pi$ is projective. Let $J_1$ and $J_2$ be subsets of $\{1, \dots, n\}$ so that $A_{J_1}$ and $A_{J_2}$ are minimal. Then $|J_1| = |J_2| = j$ for some $j$, and there are points $p_1 \in A_{J_1}$ and $p_2 \in A_{J_2}$ which are connected by a chain of rational curves $C_1, \dots, C_m$ contained in in $A^{j-1}$, each having the property that $C_i$ intersects $A^j$ transversally in a distinct pair of points.
\end{theorem}
As a consequence of Lemma \ref{lemma:deftori} and Theorem \ref{lemma:kollar}, all profound tori in a projective Calabi--Yau degeneration are homotopic to one another. This allows us to conclude the proof of Theorem \ref{thm:mainintro} by following the arguments in Section \ref{sect:consolidation}.
\begin{theorem}\label{thm:maincy}
Let $(\mathscr{X},\pi)$ be a Calabi--Yau degeneration and that $\pi$ is projective. Then if $J$ is a minimal stratum in $X_0$, the torus $\T_J$ of dimension $\delta = |J| - 1$ has the property that for each $k$,
\[
M_{2k-1} H^k(X_1;\mathbb{Q}) = \mathrm{ker}(H^k(X_1;\mathbb{Q}) \longrightarrow H^k(\T_J;\mathbb{Q})).
\]
\end{theorem}
\begin{remark}
Just as in the case of log Calabi--Yau pairs, if we begin with a semistable degeneration $(\mathscr{X},\pi)$ so that for any minimal strata $A_{J_1}$ and $A_{J_2}$, $|J_1| = |J_2| = j$ for some $j$, and there are points $p_1 \in A_{J_1}$ and $p_2 \in A_{J_2}$ which are connected by a chain of rational curves $C_1, \dots, C_m$ in $A^{j-1}$, each having the property that $C_i$ intersects $A^j$ transversally in a distinct pair of points, the conclusions of Theorem \ref{thm:maincy} continue to hold.
\end{remark}

\section{P=W type results for surfaces}\label{sect:cons}

In this section, we will show that there are P=W type identities which appear for log Calabi--Yau pairs $(X,Y)$ where $X$ is a rational surface, and $Y$ is a reduced, simple normal crossings anticanonical divisor in $X$, and when $S$ is a K3 surface.


\subsection{P=W for maximal log Calabi--Yau pairs in dimension 2}\label{sect:cypairs}
Let $(X,Y)$ be a pair consisting of a rational surface $X$ and $Y$ a reduced nodal anticanonical divisor. The component pieces of the following result are well known but we will explain it for the sake of completeness.
\begin{theorem}\label{thm:symhk}
For some symplectic form on $X$, there is a Lagrangian 2-torus fibration on $X \setminus Y$ whose generic fiber is homotopic to a profound torus and whose only singular fibers are nodal 2-tori.
\end{theorem}
\begin{proof}
This is essentially a consequence of results of Symington \cite{sym} and Gross, Hacking and Keel \cite{ghk}. Let $(X,Y)$ be a pair consisting of a smooth rational surface $X$ and a nodal anticanonical divisor $Y$. Then according to \cite[Proposition 1.3]{ghk}, $X\setminus Y$ can be constructed from a toric variety.

To any fan $\Sigma$ in $\mathbb{R}^2$ there is a toric variety $X_\Sigma$. If $\Sigma$ is chosen so that each cone is spanned by rays in $\mathbb{Z}^2$ which generate $\mathbb{Z}^2$ then $X_\Sigma$ is smooth. Furthermore, there is an anticanonical divisor of $X_\Sigma$ which is a union of copies of $\mathbb{P}^1$ which meet transversally. These rational curves are in bijection with the rays of $\Sigma$. Let $Y_\Sigma$ denote this anticanonical divisor. Then $(X_\Sigma, Y_\Sigma)$ form a log Calabi--Yau pair which we call a {\em toric pair}. Let $(X,Y)$ is a log Calabi--Yau pair and $\dim X = 2$. If we blow up $X$ in a smooth point $p$ of $Y$, then the proper transform of $Y$ in $\mathrm{Bl}_p(X)$ is an anticanonical divisor (which is biregular to $Y$), and $(\mathrm{Bl}_p X, Y)$ is also a log Calabi--Yau pair. Gross, Hacking, and Keel \cite[Proposition 1.3]{ghk} show that if $(X,Y)$ is a log Calabi--Yau pair so that $X$ is a rational surface and $Y$ admits at least one singular point, then there is some toric pair $(X_\Sigma,Y_\Sigma)$ which can be blown up repeatedly as above to produce a log Calabi--Yau pair $(X',Y')$ and $X' \setminus Y'$ is biregular to $X\setminus Y$. Therefore, if we can produce Lagrangian torus fibrations on log Calabi--Yau pairs obtained by blowing up toric pairs, then we will obtain a Lagrangian torus fibration on $X \setminus Y$. This is precisely what Symington does.

If $X$ is a toric surface, then there is a Hamiltonian $(S^1)^2$ action on $X$, which leads to a moment map $\mu : X \rightarrow \mathbb{R}^2$ whose image is a the moment polytope $\Delta$ of $X$. The fibers over points on the interior of $\Delta$ are Lagrangian 2-tori, and fibers over points in the interiors of faces in $\Delta$ are copies of $S^1$, and fibers over vertices are simply points. From this explicit description one sees that the fibers of $\mu$ are profound tori. Furthermore, the toric boundary of $X$ maps to the boundary of $\Delta$, smooth points of the boundary of $X$ mapping to points on the interior of a face of $\Delta$. Therefore, our claim is true for toric surfaces.

This is called {\em toric fibration} on $X$. A surface which admits such a fibration, extending to the anticanonical boundary in this way, but in addition may admit nodal fibers, is called an {\em almost toric fibration}. In \cite[Section 5.4]{sym}, Symington explains that if $X$ admits an almost toric fibration, and $X$ is blown up in a smooth point in its boundary, then the blow up can also be equipped with an almost toric fibration. Blowing up a smooth point in $Y$ is a surgery which leaves neighbourhoods of nodal points in $Y$ unaltered, hence Symington's surgery does not change the class of the profound torus.

Therefore, combining Symington's construction with \cite[Proposition 1.3]{ghk}, we see that for some symplectic structure on $X$, there is a $(X,Y)$ an almost symplectic Lagrangian torus fibration over a polytope $\Delta$ with possibly nodal fibers and whose generic fiber is a profound torus in $X \setminus Y$.
\end{proof}
We will now let $g : X \setminus Y \rightarrow \Delta^\circ$ denote the Lagrangian torus fibration constructed in Theorem \ref{thm:symhk}. Here $\Delta^\circ$ denotes the interior of $\Delta$.
\begin{defn}
An {\em elliptic Lefschetz fibration} over the disc is an oriented $4$-manifold $M$ along with a differentiable map $f$ to the real 2-disc whose smooth fibers are 2-tori, and whose singular fibers are nodal 2-tori, in other words, near each singular point of $f$, there are differentiable complex coordinates $(z_1,z_2)$ in which $f$ may be written as $f = z_1^2 + z_2^2$. Furthermore, we assume that each fiber contains at most one singular point and that the local orientation induced by the complex coordinates agrees with the orientation on $M$.
\end{defn}

The fibrations $g: X\setminus Y \rightarrow \Delta^\circ$ constructed in Theorem \ref{thm:symhk} are examples of elliptic Lefschetz fibrations. An important invariant of an elliptic Lefschetz fibration is its monodromy representation \cite[pp. 291]{gs}, which we now explain. Let $D_\mathrm{sm}$ be the subset of $D$ made up of smooth points. Choose a base point $p$ inear the boundary of $D_\mathrm{sm}$ and a basis of counterclockwise loops $\gamma_1,\dots, \gamma_m$ starting at $p$, going around each point in $D \setminus D_\mathrm{sm}$ and intersecting only at the point $p$. By parallel transport around each loop $\gamma_i$, we obtain a diffeomorphism from $E_p = \pi^{-1}(p)$ to itself, which we call $T_i$. The diffeomorphisms $T_1,\dots, T_m$ determine a representation,
\[
\pi_1(D_\mathrm{sm},p) \longrightarrow \mathrm{Diff}^+(E_p) \longrightarrow \mathrm{Map}(E_p) \cong \mathrm{SL}_2(\mathbb{Z}).
\]
Here, $\mathrm{Diff}^+(E_p)$ is the group of orientation preserving diffeomorphisms of $E_p$ and $\mathrm{Map}(E_p)$ is the mapping class group of $E_p$. Since $g$ has fibers with at worst ordinary double points, each path $\gamma_i$ is associated with a distinguished vanishing cycle, that is, the homology class of the copy of $S^1$ in $E_p$ which contracts to a point if we approach the point in $D \setminus D_\mathrm{sm}$ contained in $\gamma_i$. The diffeomorphisms $T_i$ are Dehn twists around the vanishing cycles of the singular fibers encircled by $\gamma_i$. There is a transitive braid group action on all such collections of loops whose action on $(T_1,\dots, T_m)$ can be described \cite[pp. 297]{gs}. The braid group orbits of the mapping class group images of $T_1,\dots, T_m$ classify $M \rightarrow D$ up to diffeomorphism. Choosing a primitive class $u \in H_1(E_p;\mathbb{Z})$, one may alter $(M,f)$ by attaching a neighbourhood of a nodal fiber whose monodromy diffeomorphism is the positive Dehn twist around $u$, and thus $u$ is the vanishing cycle at the new critical value. If we choose a basis $\alpha, \beta$ of $H_1(E_p;\mathbb{Z})$ in which $u = s \alpha +t  \beta$, the monodromy transformation of this new family is given by the matrix
\[
R_{s,t} = \left( \begin{matrix} 1-st & s^2 \\ -t^2 & 1 +st \end{matrix} \right).
\]
The matrices $R_{s,t}$ are the $\mathrm{SL}_2(\mathbb{Z})$ conjugates of the matrix $R_{1,0}$. Our goal is now to show that $X\setminus Y$ admits a complex structure so that the map $g : X\setminus Y \rightarrow \Delta^\circ$ is a holomorphic map. 
\begin{proposition}\label{prop:complexstr}
Any elliptic Lefschetz fibration $g: M \rightarrow D$ over the disc can be embedded into an elliptic surface $\mathscr{E}$ with elliptic fibration $f: \mathscr{E} \rightarrow \mathbb{P}^1$ so that
\begin{equation}\label{eq:embell}
\begin{tikzcd}
M \ar[r] \ar[d,"g"] & \mathscr{E} \ar[d,"f"] \\
D \ar[r,"i"] & \mathbb{P}^1 
\end{tikzcd}
\end{equation}
commutes for some inclusion $i : D \rightarrow \mathbb{P}^1$.
\end{proposition}
\begin{proof}
Our goal is to show that $M$ can be embedded in an elliptic Lefschetz fibration $f:\mathscr{E} \rightarrow S^2$ as in (\ref{eq:embell}). Then we may use a result of Livne and Moishezon to conclude that $\mathscr{E}$ is in fact an algebraic surface and $g$ is an elliptic fibration.

We let $T_\infty = T_1\dots T_k$. We will show that we may find matrices $S_1,\dots, S_b$ which are conjugates of $M_{1,0}$ so that $T_\infty S_1 \dots S_b = \mathrm{id}_{2\times 2}$. Therefore we can attach handles to $M$ according to $S_1,\dots, S_b$ so that the resulting $4$-manifold $M'$ still admits an elliptic Lefschetz fibration, whose monodromy around the boundary of the disc is trivial. Therefore we may extend our fibration to a fibration over $S^2$. 

We would like to show that $T_\infty^{-1}$ can be factored as a product of positive Dehn twists. To show this, it is enough to show that each $T_i^{-1}$ can be written as a product of conjugates of matrices of the form $R_{s,t}$. We may write $T_i = L_i R_{1,0} L_i^{-1}$ for some matrix $L_i$ in $\mathrm{SL}_2(\mathbb{Z})$, therefore $T^{-1}_i = L_i R_{1,0}^{-1} L_i^{-1}$, so it is sufficient to write $R_{1,0}^{-1}$ as a product of conjugates of $R_{1,0}$, but this can be done without much difficulty. For instance, 
\[
R_{0,1}R_{3,1}R_{6,1}R_{1,0}^8 = \left(\begin{matrix}1 & 0 \\ -1 & 0 \end{matrix}\right) \left( \begin{matrix} -2 & 9 \\ -1 & 4 \end{matrix} \right) \left(\begin{matrix} -5 & 36 \\ -1 & 7 \end{matrix} \right) \left( \begin{matrix} 1 & 1 \\ 0 & 1 \end{matrix} \right)^8 = \left( \begin{matrix} 1 & -1 \\ 0 & 1 \end{matrix} \right).
\]
To each of the elements $S_i$ of $\mathrm{Map}(E_p)$, we may construct an elliptic Lefschetz fibration $M_i$ over the disc $D_i$ with fibration map $g_i : M_i \rightarrow D_i$ whose geometric monodromy is $S_i$ and whose smooth fiber over a point in the boundary of $D_i$ is $E_p$. Attaching the manifolds $M_i$ to $M$ along a neighbourhood of $E_p$, we construct a differentiable manifold $M'$ with an elliptic Lefschtez fibration $g'$ over the disc, which contains $M$ as a submanifold in such a way that $g'|_M = g$. The monodromy of $g' : M' \rightarrow D$ around the boundary of $D$ given by $\mathrm{id}_{2\times 2}$. A Lefschetz fibration over $S^1$ is determined up to diffeomorphism by its monodromy automorphism in $\mathrm{Map}(E_p)$. Therefore, the monodromy automorphism moving counterclockwise around the boundary of $D$ agrees with the monodromy automorphism of $E_p \times D$ moving counterclockwise around the boundary of $D$.

Therefore we may glue $E_p \times D$ to $M'$ in such a way that $g$ extends to an elliptic Lefschetz fibration $f: \mathscr{E} \rightarrow S^2$. A theorem of Livne and Moishezon (published as an appendix in \cite[Appendix II]{moish}) says that elliptic Lefschetz fibrations over the sphere are uniquely determined up to diffeomorphism by the number of singular fibers, which must be a multiple of 12. Since for any $n \geq 0$ , there is an algebraic surface fibered over $\mathbb{P}^1$ with $12n$ nodal fibers, it follows that $M'$ embeds as a subset of an algebraic surface, as does $M$.
\end{proof}

Now we would like to compute the perverse Leray filtration of an elliptic Lefschetz fibration over the complex disc. The following proposition is helpful.

\begin{proposition}\label{prop:perverse}
Let $V$ be a K\"ahler surface and assume that $g: V \rightarrow D$ is a proper map to any Riemann surface $D$. Then the perverse Leray filtration on cohomology with respect to $g$ has the property that 
\[
P_{k-1}H^k(V;\mathbb{Q}) = \mathrm{ker}(H^k(V;\mathbb{Q}) \rightarrow H^k(F;\mathbb{Q}))
\]
for $F$ a smooth fiber of $g$.
\end{proposition}
\begin{proof}
According to Saito \cite{saito}, the decomposition theorem remains true in this context. Following \cite[Example 1.8.4]{dcm2}, we then have the following expression for $R\pi_*\mathbb{Q}_{V}$,
\begin{equation}\label{eq:decompthm}
Rg_*\mathbb{Q}_{V}[2] \cong \mathrm{IC}_D(R^2g_*\mathbb{Q}_{V}) \oplus \mathrm{IC}_D(R^1g_*\mathbb{Q}_{V}) \oplus T_\Sigma \oplus \mathrm{IC}_D(R^0g_*\mathbb{Q}_D)
\end{equation}
where $T_\Sigma$ is a skyscraper sheaf supported on the critical values of $g$ and whose rank is given by $\rho_p -1$ where $\rho_p$ denotes the number of irreducible components of $g^{-1}(p)$. In this case, the intersection cohomology sheaves can be identified with $j_*j^*R^ig_*\mathbb{Q}[i]$ where $j : D_\mathrm{sm} \hookrightarrow D$ is the embedding. Therefore,
\[
\mathrm{IC}_D(R^2g_*\mathbb{Q}_V) \cong \mathbb{Q}_D[2], \, \, \mathrm{IC}_D(R^1g_*\mathbb{Q}_V) \cong j_*j^*R^1g_*\mathbb{Q}_V[1], \, \, \mathrm{IC}_D(R^0g_*\mathbb{Q}_V) \cong \mathbb{Q}_D.
\]
We may then compute the first piece of the perverse Leray filtration to be
\[
P_{k-1}H^k(V;\mathbb{Q}) = \ker( H^k(V;\mathbb{Q}) \longrightarrow H^0(D,\mathrm{IC}_D(R^kg_*\mathbb{Q}_V))).
\]
By the argument above, $H^0(D,\mathrm{IC}_D(R^ig_*\mathbb{Q}_V))$ is precisely the monodromy fixed part of the cohomology of a smooth fiber, and $H^i(V;\mathbb{Q}) \rightarrow H^0(D,\mathrm{IC}_D(R^ig_*\mathbb{Q}_V))$ is identified with the pullback in cohomology to any smooth fiber. Therefore
\[
P_{k-1}H^k(V;\mathbb{Q}) = \ker( H^k(V;\mathbb{Q}) \longrightarrow H^k(F;\mathbb{Q})).
\]
\end{proof}
Finally, we prove the main theorem of this section.
\begin{theorem}\label{thm:pw}
Let $g: X \setminus Y \rightarrow \Delta^\circ$ be the Lagrangian torus fibration described in the proof of Theorem \ref{thm:symhk}, let $P_\bullet$ be the perverse Leray filtration on $H^*(X\setminus Y;\mathbb{Q})$ associated to $g$, and let $W_\bullet$ be the weight filtration on $H^*(X\setminus Y;\mathbb{Q})$. Then
\[
P_mH^k(X\setminus Y;\mathbb{Q}) = W_{2m}H^k(X\setminus Y;\mathbb{Q}) = W_{2m+1}H^k(X\setminus Y;\mathbb{Q}).
\]
for all $k, m$.
\end{theorem}
\begin{proof}
There is a spectral sequence which allows us to compute the weight filtration on the mixed Hodge structure of $H^i(X\setminus Y;\mathbb{Q})$ (see e.g. \cite[Lemma 3.2.7]{deligne}, \cite[Proposition 8.34]{voisin}). This sequence has $E_1$ term given by
\[
E_1^{p,q} = H^{2p+q}(Y^{-p};\mathbb{Q}).
\] 
The spectral sequence above degenerates at the $E_2$ term to $H^{p+q}(X \setminus Y;\mathbb{Q})$ and $E_2^{p,q} \cong \Gr^W_qH^{p+q}(X\setminus Y;\mathbb{Q})$.

In the situation at hand, $Y^0 \cong X$ is rational,  $Y^1$ is a union of rational curves. As a consequence, we can deduce that $W_1H^1(X \setminus Y;\mathbb{Q}) = W_3H^2(X\setminus Y;\mathbb{Q}) = 0$ and $H^3(X\setminus Y;\mathbb{Q}) \cong H^4(X\setminus Y ;\mathbb{Q}) = 0$. Therefore, each cohomology group $H^i(X\setminus Y;\mathbb{Q})$ has the property that
\[
W_{2j}H^i(X\setminus Y;\mathbb{Q}) = W_{2j+1}H^i(X\setminus Y;\mathbb{Q})
\]
for all $i$ and $j$, and the only possible non-trivial weight-graded pieces in cohomology of $X\setminus Y$ are
\[
\Gr^W_4H^2(X\setminus Y;\mathbb{Q}), \,\Gr_2^WH^2(X\setminus Y;\mathbb{Q}), \, \Gr^W_2H^1(X\setminus Y;\mathbb{Q}), \, \Gr_0H^0(X\setminus Y;\mathbb{Q}).
\]
Therefore \ref{thm:mainintro}(1) determines the weight filtration on the cohomology of $X \setminus Y$ entirely. In other words, we have
\begin{equation}\label{eq:weight}
W_{2k-1}H^k(X\setminus Y ;\mathbb{Q})= W_{2k-2}H^k(X\setminus Y;\mathbb{Q}) = \ker(H^k(X\setminus Y;\mathbb{Q}) \rightarrow H^k(\T ;\mathbb{Q}))
\end{equation}
for $\T$ a profound torus and
\[
W_{2k-3}H^k(X\setminus Y;\mathbb{Q}) = 0, \quad W_{2k}H^k(X \setminus Y;\mathbb{Q}) = H^k(X \setminus Y;\mathbb{Q})
\]
for all $k$.

Now we notice that, using the decomposition theorem as described in the proof of Proposition \ref{prop:perverse} and Proposition \ref{prop:complexstr}, the perverse Leray filtration on $H^i(X\setminus Y;\mathbb{Q})$ with respect to $g$ has at most two steps. In this case, the lowest step is determined by 
\begin{equation}\label{eq:perverse}
P_{i-1}H^i(X\setminus Y;\mathbb{Q}) = \ker( H^i(X\setminus Y;\mathbb{Q}) \rightarrow H^i(F ;\mathbb{Q}))
\end{equation}
where $F$ is a smooth fiber of $g$, and
\[
P_{k-2}H^k(X\setminus Y;\mathbb{Q}) = 0, \quad P_kH^k(X\setminus Y;\mathbb{Q}) = H^k(X\setminus Y;\mathbb{Q}).
\] 
By Theorem \ref{thm:symhk}, $F$ is homotopic to a profound torus, so comparing (\ref{eq:weight}) and (\ref{eq:perverse}) proves the result.
\end{proof}
\begin{remark}
For the proof of Theorem \ref{thm:pw}, we really only require that a fiber $F$ of $g$ be {\em homologous} to $\T$, so in fact our results are stronger than necessary for Theorem \ref{thm:pw}. Szab\'o \cite{szabo1} has recently proven the P=W conjecture for certain surfaces. The fact that the torus $\T$ is homologous to a fiber of the Hitchin map under the nonabelian Hodge correspondence is a part of his proof.
\end{remark}
\begin{remark}
In \cite{zh}, Zhang proves a similar result for a class of log Calabi--Yau surfaces. Zhang constructs a collection of cluster surfaces and proves that they admit fibrations over the disc whose generic fiber is a 2-torus and whose singular fibers are degenerate elliptic curves which have several nodes (fibers of Kodaira type $I_n$). Zhang then proves that these differentiable torus fibrations admit complex structures and that the weight filtration on their cohomology corresponds to the perverse Leray filtration of the torus fibration. It might be possible to use Theorem \ref{thm:mainintro}(1) to simplify his results.
\end{remark}


\subsection{P=W for K3 surfaces}

In this section, we will prove a P=W type result for K3 surfaces. First, we remark that the course we take in this section is slightly different from that of the previous section. Our goal here is to prove, fully, the ``P=W conjecture for K3 surfaces''. This conjecture, and probably the result itself, will admit generalization to any hyperk\"ahler manifold which admits a holomorphic Lagrangian torus fibration.

We begin with an elliptic fibration on a K3 surface with section, $g : S \rightarrow \mathbb{P}^1$. Let $\beta$ denote the cohomology class of a fiber of $g$. The class $\beta$ has the property that $\langle \beta, \beta \rangle  = 0$ where $\langle \bullet, \bullet \rangle$ denotes the intersection pairing on $H^2(S;\mathbb{Q})$. We will use the notation $\Lambda_{\mathrm{K3}}$ to denote the lattice $E_8(-1)^{\oplus 2} \oplus U^{\oplus 3}$. This lattice is isomorphic to $H^2(S;\mathbb{Z})$ equipped with the pairing $\langle \bullet, \bullet \rangle$. Here $E_8(-1)$ is the negative definite lattice associated to $E_8$ root system and $U$ is the lattice is the unique even indefinite lattice of rank 2.

For any $\rho \in H^2(S;\mathbb{Z})$ with the property that $\langle \rho, \beta \rangle  = 0$ and $\langle \rho, \rho \rangle > 0$, we may construct an operator which we denote $N_{\beta,\eta}$ in $\mathrm{End}(H^2(S;\mathbb{Z}))$ by letting
\[
N_{\beta,\rho} : x \mapsto \langle x, \beta \rangle \rho - \langle x , \rho \rangle \beta
\]
(see \cite[Lemma 1.1]{fs}). One may define a filtration $M'_\bullet$ on $H^2(S;\mathbb{Z})$ from $N_{\beta,\rho}$ by letting
\[
M'_0= M'_1 = \mathrm{im}(N^2_{\beta,\rho}) , \qquad  M'_2 = M'_3 = \mathrm{ker}(N^2_{\beta,\rho}), \qquad M'_4 = H^2(S;\mathbb{Z}).
\]

\begin{proposition}\label{prop:weightfilt}
The following statements are true.
\begin{enumerate}
\item The span of $\beta$ is $M'_0$, and $M'_2 = \beta^\perp$ (hence $M'_\bullet$ is independent of $\rho$).
\item The filtration $M'_\bullet$ is the monodromy weight filtration associated to the operator $T_{\eta,\rho} = \exp(N_{\eta,\rho})$ (c.f. Definition \ref{defn:mwf}).
\end{enumerate}
\end{proposition}
\begin{proof}
This is a direct check.
\begin{equation}\label{eq:N2}
N_{\beta,\rho}^2(x) = \langle \langle x, \beta \rangle \rho - \langle x , \rho \rangle \beta, \beta \rangle \rho - \langle \langle x, \beta \rangle \rho - \langle x , \rho \rangle \beta , \rho \rangle \beta = \langle x, \beta \rangle \langle \rho ,\rho \rangle \beta
\end{equation}
from which (1) follows easily. Finally we show that this is in fact the monodromy weight filtration associated to $N_{\beta, \rho}$. It is clear that the image of $N_{\beta,\rho}$ is contained in the span of $\rho$ and $\beta$, which is by definition in $W_2$. We note that $N_{\beta,\rho}(\rho) = - \langle \rho ,\rho \rangle \beta$ is in $M'_0$, therefore $N_{\beta,\rho}(M'_2) \subset M'_0$. To see that $N_{\beta, \rho}^2$ induces an isomorphism between $\Gr^{M'}_4$ and $\Gr^{M'}_0$, note that $\Lambda_{\mathrm{K3}}$ is nondegenerate, hence $\rank \beta^\perp = 21$. We then see that $\eta \notin M'_3$, hence $\eta$ spans $\Gr^{M'}_4$. Furthermore, $N_{\beta,\rho}^2(\eta) \neq 0$, hence $N_{\beta,\rho}^2$ induces the required isomorphism. 
\end{proof}
\begin{proposition}\label{prop:pervfilt}
Let $P_\bullet$ denote the perverse Leray filtration associated to $g : S \rightarrow \mathbb{P}^1$. Then $P_i H^2(S;\mathbb{Q}) = M'_{2i}H^2(S;\mathbb{Q})$ for all $i$.
\end{proposition}
\begin{proof}
First, Proposition \ref{prop:perverse} tells us that $P_1H^2(S;\mathbb{Q})$ is the kernel of the restriction to a fiber $F$ of $g$. Since $\beta$ is the class of a fiber, it follows that $\beta^\perp = P_1H^2(S;\mathbb{Q}) = W_2H^2(S;\mathbb{Q})$. On the other hand, the decomposition theorem (see (\ref{eq:decompthm})) tells us that $P_0H^2(S;\mathbb{Q})$ is the image of $H^2(\mathbb{P}^1, R^0g_*\mathbb{Q}_S)$, which is spanned by the class of a fiber, $\beta$. Therefore, $P_0H^2(S;\mathbb{Q}) = M'_0H^2(S;\mathbb{Q})$.

Alternately, the perverse Leray filtration behaves nicely with respect to the Poincar\'e pairing and that, in particular, $P_0H^2(S;\mathbb{Q})^\perp = P_1H^2(S;\mathbb{Q})$ \cite[pp.1115-18 (32)]{williamson}. By (\ref{eq:N2}), $M'_2H^2(S;\mathbb{Q})$ is precisely the orthogonal complement of the span of $\beta$, which is $M'_0H^2(S;\mathbb{Q})$. Since $M'_2H^2(S;\mathbb{Q}) = P_1H^2(S;\mathbb{Q})$ it followws that $P_0H^2(S;\mathbb{Q}) = M'_0H^2(S;\mathbb{Q})$ as well.
\end{proof}

We would now like to identify $M'_\bullet$ with the monodromy weight filtration of degeneration of hyperk\"ahler manifolds. This can be done using the global Torelli theorem for K3 surfaces. A good reference for this material is \cite{dolgachev}. 

The period space of $m$-polarized K3 surfaces is a type IV symmetric domain which can be expressed in the following way. Let us choose a primitive element $\alpha$ of $\Lambda_\mathrm{K3}$ whose square is $2m > 0$. Then the period space of marked K3 surfaces for which $\alpha$ is pseudoample is given by
\[
\mathscr{P}_\alpha = \{ \tau \in \mathbb{P}(\alpha^\perp \otimes \mathbb{C}) : \langle \tau, \tau \rangle = 0 , \langle \tau, \overline{\tau} \rangle > 0 \}.
\]
This domain has two components which we refer to as $\mathscr{P}_\alpha^\pm$. Then we let $\mathrm{O}^+(\Lambda_{\mathrm{K3}},\alpha)$ be the subgroup of $\mathrm{O}(\Lambda_{\mathrm{K3}})$ which fixes $\alpha$ and does not exchange the components $\mathscr{P}_\alpha^\pm$. We define
\[
\mathscr{M}_\alpha = \mathscr{P}^+_\alpha/ \mathrm{O}^+(\Lambda_{\mathrm{K3}}, \alpha)
\]
The domain $\mathscr{M}_\alpha$ admits a compactification called the Baily--Borel compactification, which is denoted $\mathscr{M}_\alpha^\mathrm{BB}$. It is known that $\mathscr{M}_\alpha^\mathrm{BB} \setminus \mathscr{M}_\alpha$ admits a stratification into points (called type III boundary components or cusps) and curves (called type II boundary components) which are in bijection with rank 1 isotropic subspaces of $\alpha^\perp$ up to automorphism by $\mathrm{O}^+(\Lambda_\mathrm{K3},\alpha)$ and rank 2 totally isotropic subspaces of $\alpha^\perp$ up to automorphism by $\mathrm{O}^+(\Lambda_\mathrm{K3},\alpha)$ respectively. Therefore if we choose a primitive class $\beta \in \Lambda_\mathrm{K3}$ so that $\langle \beta, \beta \rangle = 0$ and $\langle \beta, \alpha \rangle = 0$, then $\beta$ determines a cusp $c_\beta$ of $\mathscr{M}_\alpha^\mathrm{BB}$. 

The following result follows from arguments of Soldatenkov, but our statement is slightly different from the those in \cite{sold}. 
\begin{theorem}[{Soldatenkov, \cite[Section 4]{sold}}]\label{thm:construction}
There is a projective family of K3 surfaces $\pi: \mathscr{S}^* \rightarrow \Delta^*$ whose monodromy automorphism on $H^2(S_1;\mathbb{Q})$ has logarithm which is a multiple of $N_{\beta,\rho}$.
\end{theorem}
\begin{proof}
Associated to the operator $N_{\beta,\rho}$, one constructs a family of nilpotent orbits in the compact dual $\mathscr{P}_\alpha^\vee$ of $\mathscr{P}_\alpha$. Recall that $\mathscr{P}_\alpha^\vee$ is the subset of the quadric $\tau \in \mathbb{P}(\alpha^\perp \otimes \mathbb{C}), \langle \tau , \tau \rangle  = 0$ which can be obtained as $\exp(2\mathtt{i}\pi t N_{\beta,\rho})x$ for some $t \in \mathbb{C}$ and $x \in \mathscr{P}_\alpha$. In \cite[Lemma 4.4]{sold}. For some fixed $x\in \mathscr{P}^\vee_\alpha$, the subset given by 
\[
\{ \exp(2\mathtt{i}\pi t N_{\beta,\rho})x : t \in \mathbb{C} \}
\]
is called a {\em nilpotent orbit} of $N_{\beta,\rho}$. Soldatenkov shows that the collection of all nilpotent orbits associated to $N_{\beta,\rho}$ is open and nonempty, hence their intersection with $\mathscr{P}_\alpha$ is open and nonempty as well. 

In \cite[Lemma 4.5]{sold}, using the Torelli theorem for K3 surfaces, Soldatenkov shows that there is some projective family $\mathscr{Y} \rightarrow \mathscr{B}$ so that the period map $\mathrm{per} : \mathscr{B} \rightarrow \mathscr{M}_\alpha$ is dominant. Therefore, there is some nilpotent orbit $\mathscr{N}$ whose intersection with a neighbourhood of the cusp $c_\beta$ is a disc $\Delta^*$. By work of Borel, this map extends to a map from $\Delta$ to $\mathscr{M}_\alpha^\mathrm{BB}$. Thus we can construct a family of K3 surfaces $\mathscr{S}^* \rightarrow \Delta^*$ whose period map is a neighbourhood of the origin in a nilpotent orbit associated to $N_{\beta,\rho}$. Hence the monodromy of $\mathscr{S}$ is, after finite order base change, a power of $\exp(N_{\beta,\rho})$ \cite[Theorem 4.6]{sold}. 
\end{proof}
\begin{theorem}\label{thm:pwk3}
Let $g: S \rightarrow \mathbb{P}^1$ be an elliptic fibration with section on a K3 surface. Then there is a semistable degeneration of K3 surfaces $(\mathscr{S},\pi)$ and a diffeomorphism between $S$ and $S_1$ so that $P_jH^i(S;\mathbb{Q}) = W_{2j}H^i(S_1;\mathbb{Q})= W_{2j+1}H^i(S_1;\mathbb{Q})$ for all $i$ and $j$.
\end{theorem}
\begin{proof}
In the cases where $i = 0,1,3,4$, there is nothing to prove. In the case where $i = 2$, Theorem \ref{thm:construction} shows that there is some projective family $\mathscr{S}^* \rightarrow \Delta^*$ whose log of monodromy is a multiple of $N_{\beta,\rho}$. Using Proposition \ref{prop:weightfilt}, it follows that the monodromy weight filtration associated to this family is $M_\bullet$. Then Proposition \ref{prop:pervfilt} identifies the monodromy weight filtration of $(\mathscr{S},\pi)$ with the perverse Leray filtration of $g$.
\end{proof}
\begin{remark}
In the case where $g$ has only Kodaira type $I_1$ fibers, this theorem was essentially proved by Gross \cite[Section 7]{gross}. Gross showed that, given Lagrangian torus fibration with section on a K3 surface, whose fibers are at worst nodal 2-tori, one can construct an operator that looks like the monodromy operator of a degeneration, and the {\em Leray filtration} of this Lagrangian torus fibration agrees with the monodromy weight filtration induced by his monodromy operator. In the case where all fibers are of type $I_1$, the Leray and perverse Leray filtrations coincide \cite[\S 1.4.3]{s-y}, so our result coincides with that of Gross. 
\end{remark}

We will now briefly discuss how this should work in the case where $M$ is an arbitrary hyperk\"ahler manifold. We let $\ell : M \rightarrow B$ be a holomorphic Lagrangian torus fibration on $M$. Then the intersection pairing $\langle \bullet, \bullet \rangle$ should be replaced with the Beauville--Bogomolov--Fujiki form, $q(\bullet,\bullet)$. Under this identification, the class $\beta$ in $H^2(M;\mathbb{Q})$ given by the pullback of an ample divisor on $B$ has square 0. Choosing an element $\rho \in H^2(M;\mathbb{Q})$ which is orthogonal to $\beta$ with respect to $q$, we obtain a pair of operators on $H^*(M;\mathbb{Q})$ given by cup product with $\beta$ and $\rho$ respectively, and denoted $L_\beta, L_\rho$. The operator $L_\beta$ admits and adjoint $L_\beta^\dagger$, so we may define $N_{\beta,\rho} = [L_\rho, L_\beta^\dagger]$. Then all of the results above should be true in this more general context; for instance, Theorem \ref{thm:construction} can be applied verbatim to arbitrary hyperk\"ahler manifolds. 
\begin{remark}
A consequence of the P=W conjecture for compact hyperk\"ahler manifolds is that if $(M,\pi)$ is a type III degeneration of compact hyperk\"ahler manifold, the limit mixed Hodge structure should have the so-called ``curious hard Lefschetz property'' (see e.g. \cite[(1.1.2)]{dchm}). A proof of this fact will be given in \cite{ha1}.
\end{remark}

\bibliographystyle{plain}
\bibliography{wflcyref}

\end{document}